\documentclass[letterpaper,11pt]{article}

\usepackage{float}

\usepackage[margin=1in]{geometry}  
\usepackage{xspace,exscale,relsize}
\usepackage{fancybox,shadow}
\usepackage{graphicx}
\usepackage{color}
\usepackage{amsthm,amsfonts}
\usepackage{amssymb}
\usepackage{authblk}
\usepackage[title]{appendix}
\usepackage{caption}
\usepackage{extarrows}

\usepackage[noadjust]{cite}

\usepackage{amsmath,mathrsfs}
	\makeatletter
	\let\over=\@@over \let\overwithdelims=\@@overwithdelims
	\let\atop=\@@atop \let\atopwithdelims=\@@atopwithdelims
  	\let\above=\@@above \let\abovewithdelims=\@@abovewithdelims
  	\makeatother
\usepackage{rotating}
\usepackage{multirow}

\usepackage{ifpdf}

\usepackage{subfigure}
\usepackage{psfrag}

\usepackage{prettyref}

\usepackage{tikz}
\usetikzlibrary{arrows}
\tikzstyle{int}=[draw, fill=blue!20, minimum size=2em]
\tikzstyle{dot}=[circle, draw, fill=blue!20, minimum size=2em]
\tikzstyle{init} = [pin edge={to-,thin,black}]

\usepackage[
            CJKbookmarks=true,
            bookmarksnumbered=true,
            bookmarksopen=true,
            colorlinks=true,
            citecolor=red,
            linkcolor=blue,
            anchorcolor=red,
            urlcolor=blue
            ]{hyperref}

\usepackage[all]{xy}

\usepackage{paralist}
\usepackage{enumitem}

\usepackage{mathtools}

\usepackage{ifthen}
\newboolean{aos}
\setboolean{aos}{FALSE}

\ifx\eqref\undefined
	\newcommand{\eqref}[1]{~(\ref{#1})}
\fi
\ifx\mod\undefined
	\def\mod{\mathop{\rm mod}}
\fi

\usepackage{bm}

\newcommand{\norm}[1]{{\left\Vert #1 \right\Vert}}

\def\argmin{\mathop{\rm argmin}}
\def\argmax{\mathop{\rm argmax}}

\def\exp{\mathop{\rm exp}}

\makeatletter
\def\bbordermatrix#1{\begingroup \m@th
	\@tempdima 4.75\p@
	\setbox\z@\vbox{%
		\def\cr{\crcr\noalign{\kern2\p@\global\let\cr\endline}}%
		\ialign{$##$\hfil\kern2\p@\kern\@tempdima&\thinspace\hfil$##$\hfil
			&&\quad\hfil$##$\hfil\crcr
			\omit\strut\hfil\crcr\noalign{\kern-\baselineskip}%
			#1\crcr\omit\strut\cr}}%
	\setbox\tw@\vbox{\unvcopy\z@\global\setbox\@ne\lastbox}%
	\setbox\tw@\hbox{\unhbox\@ne\unskip\global\setbox\@ne\lastbox}%
	\setbox\tw@\hbox{$\kern\wd\@ne\kern-\@tempdima\left[\kern-\wd\@ne
		\global\setbox\@ne\vbox{\box\@ne\kern2\p@}%
		\vcenter{\kern-\ht\@ne\unvbox\z@\kern-\baselineskip}\,\right]$}%
	\null\;\vbox{\kern\ht\@ne\box\tw@}\endgroup}
\makeatother

\newcommand{\reals}{\mathbb{R}}

\newcommand{\expect}[1]{\mathbb{E}\left[#1\right]}
\newcommand{\Prob}{\mathbb{P}}
\newcommand{\prob}[1]{\mathbb{P}\left[#1\right]}

\newcommand{\pth}[1]{\left( #1 \right)}

\newcommand{\sth}[1]{\left\{ #1 \right\}}

\newcommand{\iprod}[2]{\left \langle #1, #2 \right\rangle}
\newcommand{\Iprod}[2]{\langle #1, #2 \rangle}

\definecolor{myblue}{rgb}{.8, .8, 1}
\definecolor{mathblue}{rgb}{0.2472, 0.24, 0.6} 
\definecolor{mathred}{rgb}{0.6, 0.24, 0.442893}
\definecolor{mathyellow}{rgb}{0.6, 0.547014, 0.24}

\newcommand{\red}{\color{red}}

\newcommand{\nbr}[1]{{\sf\red[#1]}}

\newcommand{\calE}{{\mathcal{E}}}

\newcommand{\calN}{{\mathcal{N}}}

\newcommand{\calP}{{\mathcal{P}}}

\newcommand{\calS}{{\mathcal{S}}}

\newrefformat{eq}{(\ref{#1})}
\newrefformat{thm}{Theorem~\ref{#1}}
\newrefformat{th}{Theorem~\ref{#1}}
\newrefformat{chap}{Chapter~\ref{#1}}
\newrefformat{sec}{Section~\ref{#1}}
\newrefformat{seca}{Section~\ref{#1}}
\newrefformat{algo}{Algorithm~\ref{#1}}
\newrefformat{fig}{Fig.~\ref{#1}}
\newrefformat{tab}{Table~\ref{#1}}
\newrefformat{rmk}{Remark~\ref{#1}}
\newrefformat{clm}{Claim~\ref{#1}}
\newrefformat{def}{Definition~\ref{#1}}
\newrefformat{cor}{Corollary~\ref{#1}}
\newrefformat{lmm}{Lemma~\ref{#1}}
\newrefformat{prop}{Proposition~\ref{#1}}
\newrefformat{pr}{Proposition~\ref{#1}}
\newrefformat{property}{Property~\ref{#1}}
\newrefformat{app}{Appendix~\ref{#1}}
\newrefformat{apx}{Appendix~\ref{#1}}
\newrefformat{ex}{Example~\ref{#1}}
\newrefformat{exer}{Exercise~\ref{#1}}
\newrefformat{soln}{Solution~\ref{#1}}

\def\unifto{\mathop{{\mskip 3mu plus 2mu minus 1mu%
	\setbox0=\hbox{$\mathchar"3221$}%
	\raise.6ex\copy0\kern-\wd0%
	\lower0.5ex\hbox{$\mathchar"3221$}}\mskip 3mu plus 2mu minus 1mu}}

\ifx\lesssim\undefined
\def\simleq{{{\mskip 3mu plus 2mu minus 1mu%
	\setbox0=\hbox{$\mathchar"013C$}%
	\raise.2ex\copy0\kern-\wd0%
	\lower0.9ex\hbox{$\mathchar"0218$}}\mskip 3mu plus 2mu minus 1mu}}
\else
\def\simleq{\lesssim}
\fi

\ifx\gtrsim\undefined
\def\simgeq{{{\mskip 3mu plus 2mu minus 1mu%
	\setbox0=\hbox{$\mathchar"013E$}%
	\raise.2ex\copy0\kern-\wd0%
	\lower0.9ex\hbox{$\mathchar"0218$}}\mskip 3mu plus 2mu minus 1mu}}
\else
\def\simgeq{\gtrsim}
\fi

\newtheorem{theorem}{Theorem}
\newtheorem{lemma}{Lemma} 

\theoremstyle{definition}

\newcommand{\overlap}{\mathsf{overlap}}
\newcommand{\Tr}{\mathsf{Tr}}
\newcommand{\SNR}{\mathsf{SNR}}
\newcommand{\id}{I_n}

\newcommand{\PPiXperp}{\mathcal{P}_{(\Pi X)^\perp}}
\newcommand{\PXperp}{\mathcal{P}_{ X^\perp}}
\newcommand{\PPiX}{\mathcal{P}_{\Pi X}}
\newcommand{\PX}{\mathcal{P}_{X}}

\newcommand{\fc}{\mathfrak{c}} 

\DeclareMathOperator{\E}{\mathbb{E}}
\DeclareMathOperator{\R}{\mathbb{R}}

\DeclareMathOperator{\EE}{\mathcal{E}}
\DeclareMathOperator{\NN}{\mathcal{N}}
\DeclareMathOperator{\PR}{\mathbb{P}}
\DeclareMathOperator{\eps}{\varepsilon}
\newcommand{\mybinom}[3][0.8]{\scalebox{#1}{$\dbinom{#2}{#3}$}}
\newcommand{\opnorm}[1]{\left\| #1 \right\|_{\rm op}}

\newif\ifmapx
{\catcode`/=0 \catcode`\\=12/gdef/mkillslash\#1{#1}}
\edef\jobnametmp{\expandafter\string\csname ic_apx\endcsname}
\edef\jobnameapx{\expandafter\mkillslash\jobnametmp}
\edef\jobnameexpand{\jobname}
\ifx\jobnameexpand\jobnameapx
\mapxtrue
\else
\mapxfalse
\fi

\renewcommand{\hat}{\widehat}
\renewcommand{\tilde}{\widetilde}

\begin{document}
\ifpdf
\DeclareGraphicsExtensions{.pgf,.jpg,.pdf}
\graphicspath{{figures/}{plots/}}
\fi

\title{Sharp Information-Theoretic Thresholds for Shuffled Linear Regression} 

\author{Leon Lufkin, Yihong Wu, Jiaming Xu\thanks{
L.\ Lufkin is with the Department of Statistics and Data Science, Yale University, New Haven, USA, \texttt{leon.lufkin@yale.edu}.
Y.\ Wu is with the Department of Statistics and Data Science, Yale University, New Haven, USA, \texttt{yihong.wu@yale.edu}.
J.\ Xu is with The Fuqua School of Business, Duke University, Durham NC, USA, \texttt{jiaming.xu868@duke.edu}.
}}
	

\maketitle


\begin{abstract}
This paper studies the problem of shuffled linear regression, where the correspondence between predictors and responses in a linear model is obfuscated by a latent permutation. Specifically, we consider the model $y = \Pi_* X \beta_* + w$, where 
$X$ is an $n \times d$ standard Gaussian design matrix, 
$w$ is Gaussian noise with entrywise variance $\sigma^2$, $\Pi_*$ is an unknown $n \times n$ permutation matrix, and $\beta_*$ is the regression coefficient, also unknown. 
Previous work has shown that, in the large $n$-limit, the minimal signal-to-noise ratio ($\SNR$), $\norm{\beta_*}^2/\sigma^2$, 
for recovering  the unknown permutation exactly with high probability is between $n^2$ and $n^C$
for some absolute constant $C$ and the sharp threshold is unknown even for $d=1$. We show that this threshold is precisely $\SNR = n^4$ for exact recovery throughout the sublinear regime $d=o(n)$. As a by-product of our analysis, we also determine the sharp threshold of almost exact recovery to be $\SNR = n^2$, where all but a vanishing fraction of the permutation is reconstructed.
\end{abstract}

\section{Introduction}

Consider the following linear model, where we observe
\begin{equation}
\label{eq:model}
    y = \Pi_* X \beta_* + w,
\end{equation}
Here $X \in \mathbb{R}^{n \times d}$ is the design matrix, 
 $\beta_* \in \mathbb{R}^d$ is the  unknown regression coefficient, $\Pi_*$ is an unknown $n \times n$ permutation matrix that shuffles the rows of $X$, and $w \in \mathbb{R}^n$ is observation noise.
The goal is to recover $\Pi_*$ and $\beta_*$ on the basis of  $X$ and $y$.

If $\Pi_*$ is known, \prettyref{eq:model} is the familiar linear regression. Otherwise, this problem 
is known as shuffled regression \cite{pananjady2017linear,abid2017linear}, unlabeled sensing \cite{unnikrishnan2018unlabeled,zhang2019permutation,zhang2020optimal}, or  linear regression with permuted/mismatched data \cite{slawski2019linear,mazumder2023linear,slawski2020two}, 
as the correspondence between the predictors (the rows $x_i$'s of $X$) and the responses ($y_i$'s) is lost.
As such, it is a much more difficult problem
as one needs to jointly estimate the permutation $\Pi_*$ and the regression coefficients $\beta_*$.
This is a problem of considerable theoretical and practical interest. For applications in areas such as robotics, data integration, and de-anonymization, we refer the readers to \cite[Sec.~1]{unnikrishnan2018unlabeled} and \cite[Sec.~1.1]{zhang2020optimal}.

A line of work has studied 
the minimal signal to noise ratio (SNR) that is required to reconstruct $\Pi_*$. Following \cite{pananjady2017linear,hsu2017linear}, in this paper we consider a random design $X$ with $X_{ij} \overset{\text{iid}}{\sim} \NN(0,1)$ 
and Gaussian noise $w_i \overset{\text{iid}}{\sim} \NN(0,\sigma^2)$, which are independent from each other.
Define
\begin{equation}
    \label{eq:SNR}
\SNR \triangleq \frac{\|\beta_*\|_2^2}{\sigma^2}.
\end{equation}
It is shown in \cite[Theorems 1 and 2]{pananjady2017linear} that 
for exact recovery (namely, $\hat \Pi = \Pi_*$ with probability tending to one), the required $\SNR$ is between $n^2$ and $n^C$ for some absolute constant $C$. Intriguingly, numerical simulation carried out for $d=1$ (see 
\cite[Fig.~2]{pananjady2017linear}) suggests that there is a  sharp threshold $\SNR = n^{C_0}$
for some constant $C_0$ between 3 and 5. 

The major contribution of this work is to resolve this question by showing that the sharp threshold for exact recovery is $\SNR = n^{4}$ for all dimensions satisfying $d=o(n)$.
Along the way, we also resolve the optimal threshold for achieving almost exact reconstruction, namely, 
$ \overlap(\hat{\Pi},\Pi_*) =1 - o(1)$, where 
\[
\overlap(\hat{\Pi},\Pi_*) \triangleq \frac{1}{n} \Tr(\hat \Pi^\top \Pi_*)
\]
 is the fraction of covariants that are correctly unshuffled. In other words, if $\hat \pi$ and $\pi_*$ are  permutations corresponding to $\hat{\Pi}$ and $\Pi_*$, then 
    $ \overlap(\hat{\Pi},\Pi_*) = \frac{1}{n} |\sth{ i \in [n]:\hat{\pi}(i) = \pi(i) }| $.

\section{Maximum likelihood and quadratic assignment}
\label{sec:mle_qap}
A natural idea for the joint estimation of $(\Pi_*,\beta_*)$ 
is the maximum likelihood estimator (MLE) \cite{pananjady2017linear}:
\begin{equation}
    (\hat{\Pi}, \hat{\beta}) = \argmin_{\Pi\in S_n, \beta\in\reals^d} \lVert y - \Pi X \beta \rVert_2^2,
    \label{eq:def_MLE}
\end{equation}
where $S_n$ denotes the set of all $n\times n$ permutation matrices.
Since $\beta_*$ has no structural assumptions such as sparsity, $n \geq d$ is necessary even when there is no noise and $\Pi_*$ is known.
By classical theory on linear regression, for a fixed $\Pi$
 the optimal $\beta$ for \prettyref{eq:def_MLE} is given by 
 \begin{equation}
     \label{eq:beta_Pi}
     \hat\beta_\Pi \triangleq
     (X^\top X)^{-1} X^\top \Pi^\top y
 \end{equation}
 and 
 $\lVert y - \Pi X \hat\beta_\Pi \rVert_2^2 = \|\PPiXperp y\|_2^2$,
 where
 \begin{align}
 \PPiX & \triangleq \Pi \underbrace{X (X^\top X)^{-1} X^\top}_{\triangleq \PX} \Pi^\top \label{eq:PPiX1}\\
  \PPiXperp&\triangleq I_n- \PPiX
  =\Pi \underbrace{(I_n-X (X^\top X)^{-1} X^\top)}_{\triangleq \PXperp} \Pi^\top
     \label{eq:PPiX2}
 \end{align}
are the projection matrices onto the column span of $\Pi X$ and its orthogonal complement respectively. 
 Thus the ML estimator of $\Pi_*$ can be written as\footnote{We note that although $(\hat{\Pi}, \hat{\beta})$ defined in
 \prettyref{eq:def_MLE} is the MLE for 
 $(\Pi_*,\beta_*)$, it is unclear that 
 $\hat{\Pi}$ itself (i.e., \prettyref{eq:MLE1}) is optimal (that is, minimizing the probability of error $\Prob[\hat \Pi\neq\Pi_*]$ when $\Pi_*$ is drawn uniformly at random), due to the presence of the nuisance parameter $\beta_*$.}
 \begin{equation}
    \hat{\Pi} = \argmax_{\Pi\in S_n} \lVert \PPiX y\rVert_2^2. \label{eq:MLE1}
\end{equation}
This optimization problem is in fact a special instance of the 
\textit{quadratic assignment problem} (QAP) \cite{koopmans1957assignment}:
\begin{equation}
\max_{\Pi\in S_n} \Iprod{A}{\Pi^\top B \Pi},
    \label{eq:QAP}
\end{equation}
where $A=yy^\top$ is rank-one and $B=\PX$ is a rank-$d$ projection matrix.
For worst-case instances of $(A,B)$, the QAP 
\prettyref{eq:QAP} is known to be NP-hard \cite{sahni1976p}.
Furthermore, even solving the special case \prettyref{eq:MLE1} is NP-hard provided that $d=\Omega(n)$ \cite[Theorem 4]{pananjady2017linear}.
On the positive side, for constant $d$ 
it is not hard to see that this can be solved in polynomial time.
Indeed, 
as the proof in \prettyref{sec:proofs} shows (see \cite[Sec.~2]{hsu2017linear} for a similar result),
instead of \prettyref{eq:QAP}, one can approximate the original \prettyref{eq:def_MLE} by discretizing and restricting $\beta$ to an appropriate $\delta$-net for $\delta=1/\text{poly}(n)$. Since for fixed $\beta$, \prettyref{eq:def_MLE} becomes a very special case of the  \textit{linear assignment problem} (LAP) $\max_{\Pi} \Iprod{y}{\Pi X\beta}$ which can be solved by sorting the vectors $y$ and $X\beta$, 
the discretized version of \prettyref{eq:def_MLE} can be computed in $n^{O(d)}$-time. In fact, for the special case of $d=1$, this can be made exact \cite[Theorem 4]{pananjady2017linear}.


\section{Main results}
\label{sec:results}

The following theorem on the statistical performance of the estimator \prettyref{eq:MLE1} is the main result of this paper.
\begin{theorem}
    \label{thm:main}
Fix an arbitrary $\epsilon>0$.
Assume that $d=o(n)$.
\begin{enumerate}
    \item[(a)] Exact recovery:
    If
    $\SNR \geq n^{4+\epsilon}$,
    then 
    $\Prob[\hat{\Pi} = \Pi_*]=1-o(1)$ as $n\to\infty$, where $o(1)$ is uniform in $\Pi_*$ and $\beta_*$.
    \item[(b)] Almost exact recovery:
    If 
    $\SNR \geq n^{2+\epsilon}$,
    then   $\Prob[\overlap(\hat{\Pi},\Pi_*)=1-o(1)]=1-o(1)$ as $n\to\infty$, where $o(1)$ is uniform in $\Pi_*$ and $\beta_*$.
\end{enumerate}
\end{theorem}



The positive results in \prettyref{thm:main} are in fact information-theoretically optimal. To see this, for the purpose of lower bound, consider the case where $\Pi_*$ is drawn uniformly at random and $\beta_*$ is a known unit vector. Defining $x \triangleq X\beta_* \sim \calN(0,I_n)$, we have $y = \Pi_* x + w$. 
Then the problem reduces to a special case of the linear assignment model studied in \cite{dai2019database,kunisky2022strong,WWXY22}
where the goal is to reconstruct $\Pi_*$ by  observing $x$ and $y$.\footnote{These works considered the more general setting where $x,y$ are $n\times m$ Gaussian matrices and the respective threshold for exact and almost exact reconstruction is determined to be $n^{-2/m}$ and $n^{-1/m}$ for $m=o(\log n)$.} 
Specifically,  applying 
\cite[Theorem 3]{WWXY22} for one dimension shows that exact (resp.~almost exact) reconstruction is impossible unless 
$\sigma=o(n^{-2})$ (resp.~$\sigma=o(n^{-1})$).

Next, let us comment on the role of the dimension $d$. As lower-dimensional problem instances can be embedded into higher dimensions by padding zeros to $\beta_*$, the minimum required SNR for recovery is non-decreasing in $d$. 
\prettyref{thm:main} shows the optimal thresholds for exact and approximate exact recovery are 
$\SNR=n^{4}$ and $n^{2}$ 
 in the \textit{sublinear regime} of  $d=o(n)$.
When the dimension is proportional to the sample size, say $d=\rho n$ for some constant $\rho \in (0,1)$, we conjecture that the conclusion in \prettyref{thm:main} no longer holds and the sharp threshold depends on $\rho$. In fact, 
\cite[Theorem 1]{pananjady2017linear} shows that the estimator \prettyref{eq:MLE1} achieves exact recovery provided that 
$\SNR\geq n^{C/(1-\rho)}$ for some unspecified constant $C$. On the other hand, the simple lower bound argument above does not yield any dependency on $\rho$, since it assumes $\beta_*$ is known and reduces the problem to $d=1$.
Determining the optimal threshold in the linear regime remains a challenging question.

\section{Further related work}

The model \prettyref{eq:model} has been been considered in the compressed sensing literature for zero observation noise $(\sigma = 0)$, known as the unlabeled sensing problem, with the goal of recovering $\beta_* \in \R^d$ exactly.
The work \cite{unnikrishnan2018unlabeled} showed that when the entries of $X$ are sampled iid from some continuous probability distribution, \emph{any} $\beta_*$, including adversarial instances (the so-called ``for all'' guarantee), can be recovered exactly with probability one if and only if one has $n \geq 2d$ observations.
The paper shows this using a constructive proof, but it requires a combinatorial algorithm involving exhaustive search.

Moving to the weaker ``for any'' guarantee, the works \cite{hsu2017linear,andoni2017correspondence} 
also consider the noiseless setting and
propose an efficient algorithm based on lattice reduction that recovers an arbitrary fixed $\beta_*$  with 
probability one with respect to the random design, provided that $n>d$. 
Another approach based on method of moments is proposed in \cite{abid2017linear}, where the empirical moments of $X \hat{\beta}$ are matched to those of $y$.

There is also a line of work on shuffled regression when the latent permutation is partially (or even mostly) known \cite{slawski2019linear,slawski2020two,zhang2020optimal,mazumder2023linear} that has found applications in analyzing census and climate data.
This approach permits a robust regression formulation for estimating $\beta_*$, wherein the unknown permuted data points are treated as outliers, from which $\Pi_*$ can then be estimated.

The problem of learning from shuffled data has also been considered in nonparametric settings, e.g., isotonic regression, where $y_i = f(x_i) + w_i$, for some $f:[0,1]^d\to\reals$ that is coordinate-wise monotonically increasing, and the goal is to estimate $f$.
When the $x_i$ are permuted, this problem is known as \emph{uncoupled} isotonic regression, which has been studied in \cite{rigollet2019uncoupled} for $d=1$ and in \cite{pananjady2022isotonic} for  $d>1$.

\section{Proof of \prettyref{thm:main}}
\label{sec:proofs}


Throughout the proof, we assume $\Pi_* = I_n$ without loss of generality.
The proof of Theorem \ref{thm:main} follows a union bound over $\Pi \neq I_n$ and is divided  into two parts: 
\prettyref{sec:pf-far} deals with those permutations $\Pi$ whose number of non-fixed points is at least $\eta n$ (for some $\eta=o(1)$ depending on $d$ and $\epsilon$).
\prettyref{sec:pf-near} deals with those permutations $\Pi$ whose number of non-fixed points is at most $\eta n$.

Although both \cite{pananjady2017linear} and the present paper analyze the estimator \prettyref{eq:def_MLE}, the program of our analysis deviates from that in~\cite{pananjady2017linear} in the following two aspects, both of which are crucial for determining the sharp thresholds. 

First, a key quantity appearing in the proof is the following moment generating function (MGF):
\begin{align}
\expect{\exp\left( - t \norm{X\beta_*-\Pi X\beta }^2 \right) }, \label{eq:MGF_expression}
\end{align}
for a given $\Pi$ and $\beta$, where $t \propto\frac{1}{\sigma^2}$. While similar quantities have been analyzed in~\cite{pananjady2017linear}, only a crude bound is obtained in terms of the number of fixed points of $\Pi$ (see Lemma 4 and eq.~(25-26) therein).  Instead, inspired by techniques in~\cite{WWXY22} for random graph matching, we precisely characterize \prettyref{eq:MGF_expression} in terms of the cycle decomposition of $\Pi$ and $\beta$. In particular, to determine the sharp thresholds, it is crucial to consider \textit{all}  cycle types instead of just fixed points. 

Second, recall that the MLE \prettyref{eq:def_MLE} involves a double minimization over $\Pi$ and $\beta$. While it is straightforward to solve the inner minimization over $\beta$ and obtain a closed-form expression for the optimal $\hat{\beta}_{\Pi}$ \prettyref{eq:beta_Pi}, directly analyzing the MLE with this optimal $\hat{\beta}_{\Pi}$ plugged in, namely, the QAP \prettyref{eq:QAP}, turns out to be challenging. In particular, this requires a tight control of the MGF~\prettyref{eq:MGF_expression} with $\beta$ replaced by $\hat{\beta}_{\Pi}$.
While this is doable when $\Pi$ is close to $\id$, the analysis becomes loose when $\Pi$ moves further away from $\id$ and requires suboptimally large $\SNR$. Alternatively, we do not work with this optimal $\hat{\beta}_{\Pi}$ and instead take a union bound over a proper discretization ($\delta$-net) of $\beta$. Importantly, the resolution $\delta$ needs to be carefully chosen so that the cardinality of the $\delta$-net is not overwhelmingly large compared to~\prettyref{eq:MGF_expression}. This part crucially relies on the sublinearity assumption $d=o(n)$ and the fact that $\Pi$ has at least $\eta n$ non-fixed points.

\subsection{Proof for permutations with many errors}
\label{sec:pf-far}

In this part, we focus on the permutations that are far away from the ground truth and prove that
\begin{align}
\prob{  \overlap(\hat{\Pi}, \id) \le \left( 1- \eta \right)n} \le  o(1), \label{eq:MLE_error_big}
\end{align}
for any fixed $\epsilon,$
provided that $\SNR \ge n^{2+\epsilon}$, $d=o(n)$,  $\epsilon \eta n \ge 100d$, and $\eta \ge n^{-\epsilon/10}.$
Note that here we only require
$\SNR \ge n^{2+\epsilon}$ instead of $\SNR \ge n^{4+\epsilon}$. This directly implies the desired sufficient condition for the almost exact recovery and proves Part (b) of~\prettyref{thm:main}, with an appropriate choice of $\eta=o(1).$

Let $\calS(m)$ denote the set of permutation matrices with $m$ fixed points. For a given $r$, let $B_r(\beta_*)\triangleq \{\beta: \norm{\beta-\beta_*}\le r\}$.
The following lemma shows that we can discretize $\beta$ appropriately without inflating the objective too much. 
\begin{lemma} 
\label{lem:discretization}
There exists a $\delta$-net $N_\delta(r)$ for $B_r(\beta_*)$ 
such that $|N_\delta(r)| \le (1+ 2r/\delta)^d$. Moreover, for any $\Pi$, if $ \hat{\beta}_{\Pi}  \in B_r(\beta_*)$,
$$
\min_{\beta \in N_\delta(r)} 
\norm{y-\Pi X\beta}^2
\le \min_{\beta \in B_r(\beta_*)}
\norm{y-\Pi X\beta}^2 + \opnorm{X}^2 \delta^2.
$$
\end{lemma}

Next, we introduce a set of high-probability events to facilitate our analysis of the MLE.
\begin{lemma} \label{lem:events}
Suppose $\SNR \geq 1$, $r/\delta \le n^{2},$
$\eta \ge n^{-\epsilon/10}$,
and $\epsilon \eta n \ge 100d.$
The following events hold with probability $1-o(1):$ 
\begin{align*}
  \EE_1 &\triangleq \{ 
  \lVert X \beta_* - \Pi X \beta \rVert^2 \geq n^{-2-\epsilon} \norm{\beta_\ast}^2 
  (n-n_1),  \\ 
  & \qquad \forall n_1 \le (1-\eta) n, \forall \Pi \in \calS(n_1),  \forall \beta \in N_\delta(r) \} , \\
  \EE_2 &\triangleq \{ \opnorm{X} \leq C' \sqrt{n} \}, \\
  \EE_3 &\triangleq \{ \lVert \hat{\beta}_{\Pi} - \beta_* \rVert_2 \leq c \lVert \beta_* \rVert_2, \forall \Pi \}, 
\end{align*}
for some absolute constants $C', c,$, where
$\hat{\beta}_{\Pi}$ is defined in \prettyref{eq:beta_Pi}.
\end{lemma}

Finally, we need a key lemma to bound the MGF of $\norm{X\beta_*-\Pi X\beta }^2 $. The proof crucially relies on the cycle decomposition of the permutation matrix $\Pi$. (See Appendices \ref{app:cycle} and \ref{app:MGF_bound_1} for details.)

\begin{lemma}\label{lmm:MGF_bound_1}
Suppose $\norm{\beta_*}/\sigma 
\ge n^{1+\epsilon/2}$, $\eta \ge n^{-\epsilon/10}$, and $C$ is any constant. 
Then for all sufficiently large $n$,
\begin{align*}
\sum_{n_1=0}^{(1-\eta)n} C^{n-n_1} \sum_{\Pi \in \calS(n_1) }  \expect{\exp\left( - \frac{1}{32\sigma^2} \norm{X\beta_*-\Pi X\beta }^2 \right) }
&\le n^{-\epsilon \eta n /10}.
\end{align*}
\end{lemma}

Now, we are ready to prove~\prettyref{eq:MLE_error_big}. By the definition of MLE given in \prettyref{eq:def_MLE}, 
\begin{align*}
\overlap(\hat{\Pi}, \id) \le \left( 1- \eta \right) n 
& \Rightarrow
\min_\beta \norm{y - \Pi X \beta}^2
 \le \min_\beta \norm{y - X \beta}^2\\
& \qquad \text{for some } \Pi \in \calS(n_1) \text{ and }  n_1 \le (1-\eta) n .
\end{align*}
Recall that $\hat{\beta}_\Pi= \argmin_\beta \norm{y - \Pi X \beta}^2$ and the definition of $\calE_3$. By letting 
$r=c\|\beta_\ast\|_2$, we have
\begin{align*}
\min_\beta \norm{y - \Pi X \beta}^2
 \le \min_\beta \norm{y - X \beta}^2, \, \calE_3
 & \Rightarrow \min_{\beta \in B_r(\beta_*)} 
 \norm{y - \Pi X \beta}^2 
 \le \norm{y - X \beta_*}^2=\norm{w}^2 \\
 & \Rightarrow \min_{\beta \in N_\delta(r) }
 \norm{y - \Pi X \beta}^2 
 \le  \norm{w}^2 + \opnorm{X}^2 \delta^2,
\end{align*}
where the last implication follows from 
\prettyref{lem:discretization}.
Note that for any $\beta $, 
\begin{align*}
\norm{y - \Pi X \beta}^2 
 \le  \norm{w}^2 + \opnorm{X}^2 \delta^2 
& \Rightarrow  
 \norm{X\beta_*+w - \Pi X \beta}^2 
 \le  \norm{w}^2 +  \opnorm{X}^2 \delta^2 \\ 
 & \Rightarrow 
 2\iprod{X\beta_*-\Pi X\beta}{w}
 \le - \norm{X\beta_*-\Pi X\beta }^2 
+  \opnorm{X}^2 \delta^2.
\end{align*}
Now, recalling the definitions of $\calE_1,\calE_2$, we choose $\delta = C' \sqrt{\eta/2}   n^{-1-\epsilon/2} \norm{\beta_*}$, so that on the event $\calE_1\cap \calE_2$, for all 
$\Pi \in \calS(n_1)$ and all $ n_1 \le (1-\eta) n,$
$$
\norm{X\beta_*-\Pi X\beta }^2 
\ge 2 \opnorm{X}^2 \delta^2, \forall \beta \in N_\delta(r),
$$
and hence 
\begin{align*}
& \min_{\beta \in N_\delta(r) }
 \norm{y - \Pi X \beta}^2 
 \le  \norm{w}^2 +  \opnorm{X}^2 \delta^2, \, \calE_1 \cap \calE_2 \\
& \Rightarrow \exists \beta \in N_\delta(r): 
 2\iprod{X\beta_*-\Pi X\beta}{w}
 \le - \frac{1}{2} \norm{X\beta_*-\Pi X\beta }^2 .
\end{align*}
In conclusion, we have shown that  
\begin{align*}
\overlap(\hat{\Pi}, \id) \le \left( 1- \eta \right) n, \calE_1 \cap \calE_2 \cap \calE_3 
& \Rightarrow 
2\iprod{X\beta_*-\Pi X\beta}{w}
 \le - \frac{1}{2} \norm{X\beta_*-\Pi X\beta }^2\\
& \qquad \text{for some } \Pi \in \calS(n_1), n_1 \le (1-\eta) n, \text{ and } \beta \in N_\delta(r).
\end{align*}

Now, for each fixed $\Pi$ and $\beta$, using the Gaussian tail bound, we get that 
\begin{align*}
 \prob{2\iprod{X\beta_*-\Pi X\beta}{w}
 \le - \frac{1}{2} \norm{X\beta_*-\Pi X\beta }^2} 
& \le \expect{\exp\left( - \frac{1}{32\sigma^2} \norm{X\beta_*-\Pi X\beta }^2 \right) }.
 \end{align*}
It follows from the union bound that 
\begin{align*}
&\prob{\overlap(\hat{\Pi}, \id) \le \left( 1- \eta \right) n,\calE_1 \cap \calE_2\cap \calE_3} \\
& \le |N_\delta(r) |
\sum_{n_1 \le (1-\eta) n} \sum_{\Pi \in \calS(n_1)}
\expect{\exp\left( - \frac{\norm{X\beta_*-\Pi X\beta }^2 }{32\sigma^2} \right) }.
\end{align*}
Finally, by~Lemma \ref{lem:discretization}, 
$|N_\delta(r)|\le(1+2r/\delta)^d$. Recall that we set $\delta=C' \sqrt{\eta/2}   n^{-1-\epsilon/2} \norm{\beta_*}$
and $r=c\norm{\beta_*}$ for constants $c, C'>0.$ 
Therefore, 
$$
\left|N_\delta(r)\right|
\le \left(C n^{1+\epsilon/2}/\sqrt{\eta} \right)^d
$$
for some universal constant $C>0.$
Combining the last displayed equation with~\prettyref{lmm:MGF_bound_1} yields that 
\begin{align*}
 |N_\delta(r)|
\sum_{n_1 \le (1-\eta) n} \sum_{\Pi \in \calS(n_1)}
\expect{\exp\left( - \frac{\norm{X\beta_*-\Pi X\beta }^2 }{32\sigma^2} \right) } 
&\le
\left(C n^{1+\epsilon/2}/\sqrt{\eta} \right)^d   n^{-\epsilon \eta n /10}
\le n^{-\epsilon \eta n/20},
\end{align*}
where the last inequality holds for all sufficiently large $n$ due to
the facts that 
$\epsilon \eta n \ge 100d$ and $\eta \ge n^{-\epsilon/10}$.
Finally, applying~\prettyref{lem:events}, we conclude that 
\begin{align*}
\prob{\overlap(\hat{\Pi}, \id) \le \left( 1- \eta \right) n}
& \le \prob{\overlap(\hat{\Pi}, \id) \le \left( 1- \eta \right) n,\calE_1 \cap \calE_2\cap \calE_3} \\
&\quad  + \prob{\calE_1^c} + \prob{\calE_2^c}+\prob{\calE_3^c} \\ &\le o(1).
\end{align*}

\subsection{Proof for  permutations with few errors}
\label{sec:pf-near}

In this part, we focus on the permutations that are close to the ground truth and prove that 
\begin{align}
\prob{  \left( 1- \eta \right)n \le \overlap(\hat{\Pi}, \id) \le n-2} \le n^{-\Omega(1)}, \label{eq:MLE_error_small}
\end{align}
provided that $\sigma/\norm{\beta_\ast}\le n^{-2-\epsilon}$, $\eta \le \epsilon/8$, and $d=o(n).$ 

In this case, we can no longer tolerate the $n^d$ factor arising from the discretization of the $\beta$ parameter.
To address the high-dimensional scenario where $d=o(n),$ we instead adopt the proof strategy outlined by~\cite{pananjady2017linear} 
to analyze the QAP formulation~\prettyref{eq:QAP}.
However, 
achieving the sharp threshold necessitates a more meticulous analysis than that employed by~\cite{pananjady2017linear}. 

We first state several useful auxiliary lemmas. Recall that $\calS(m)$ denotes the set of permutation matrices with 
$m$ fixed points,
and recall the projection matrices $\PPiX$ and $\PPiXperp$ as defined in~\prettyref{eq:PPiX1}--\prettyref{eq:PPiX2}.
\begin{lemma}\label{lmm:noise_projection}
Let $n \ge 2$. Define $\calE_1$ such that for all $n_1 \le n-2$ and all $\Pi \in \calS(n_1),$
$$
\norm{\PPiX(w)}^2 - \norm{\PX(w)}^2  \le 10 \sigma^2 (n-n_1) \log n. 
$$
Then $\prob{\calE_1} \ge 1-4n^{-2}.$
\end{lemma}

\begin{lemma}\label{lmm:signal_projection}
Suppose $\eta \le \epsilon/8$. Define $\calE_2$ such that for all $(1-\eta) n \le n_1 \le n-2$ and all $\Pi \in \calS(n_1),$
\begin{align*}
\norm{\PPiXperp(X\beta_*)}^2
\ge n^{-4-\epsilon} \norm{\beta_\ast}^2 (n-n_1).
\end{align*}
 Then $\prob{\calE_2} \ge 1-n^{-\epsilon/8}$.
\end{lemma}

\begin{lemma}\label{lmm:MGF_bound_2}
Suppose $\sigma/\norm{\beta_*} \le n^{-2-\epsilon/2}$, $\eta \le \epsilon/8$, and $C$ is any fixed constant. Then for all sufficiently large $n$,
\begin{align*}
 \sum_{n_1 \ge (1-\eta) n}^{n-2} C^{n-n_1} \sum_{\Pi \in \calS(n_1) }  \expect{\exp\left( - \frac{1}{32\sigma^2} \norm{\PPiXperp(X\beta_*)}^2\right)}
& \le n^{-\epsilon/8}.
\end{align*}
\end{lemma}
Now, we are ready to prove~\prettyref{eq:MLE_error_small}.
By the definition of MLE given in \prettyref{eq:def_MLE}, 
\begin{align*}
\left( 1- \eta \right) n \le \overlap(\hat{\Pi}, \id) \le n-2 
& \Rightarrow
\norm{\PPiXperp(y)}^2 \le 
\norm{\PXperp(y)}^2\\
& \qquad \text{for some } \Pi \in \calS(n_1) \text{ and } (1-\eta) n \le n_1 \le n-2.
\end{align*}
Since $\PXperp(X\beta_*) = 0,$ it follows that
\begin{align*}
& \norm{\PPiXperp(y)}^2 \le 
\norm{\PXperp(y)}^2\\
& \Leftrightarrow
\norm{\PPiXperp(X\beta_*)+\PPiXperp(w)}^2
 \le \norm{\PXperp(w)}^2 \\
 & \Leftrightarrow \norm{\PPiXperp(X\beta_*)}^2 + 2 \iprod{\PPiXperp(X\beta_*)}{\PPiXperp(w)} 
    \le \norm{\PXperp(w)}^2 - \norm{\PPiXperp(w)}^2 \\
 & \Leftrightarrow \norm{\PPiXperp(X\beta_*)}^2 + 2 \iprod{\PPiXperp(X\beta_*)}{w} \le \norm{\PPiX(w)}^2 - \norm{\PX(w)}^2.
\end{align*}
By our assumption that $\sigma/\norm{\beta_\ast} \le n^{-2-\epsilon}$, 
on event $\calE_1 \cap \calE_2$, 
for all sufficiently large $n$, all $(1-\eta) n \le n_1 \le n-2$, and all $\Pi \in \calS(n_1),$
$$
\norm{\PXperp(w)}^2 - \norm{\PPiXperp(w)}^2
\le \frac{1}{2} 
\norm{\PPiXperp(X\beta_*)}^2.
$$
Thus, on event $\calE_1 \cap \calE_2$,
\begin{align*}
\left( 1- \eta \right) n \le \overlap(\hat{\Pi}, \id) \le n-2
& \Rightarrow 
\frac{1}{2} \norm{\PPiXperp(X\beta_*)}^2 + 2 \iprod{\PPiXperp(X\beta_*)}{w}  \le 0\\
& \qquad \text{for some } \Pi \in \calS(n_1) \text{ and } (1-\eta) n \le n_1 \le n-2.
\end{align*}
By the Gaussian tail bound, 
\begin{align*}
\prob{
\frac{1}{2} \norm{\PPiXperp(X\beta_*)}^2 + 2 \iprod{\PPiXperp(X\beta_*)}{w}  \le 0
} 
& \le \expect{\exp\left( - \frac{1}{32\sigma^2} \norm{\PPiXperp(X\beta_*)}^2\right)}.
\end{align*}
Therefore, applying union-bound yields that
\begin{align*}
& \prob{\left( 1- \eta \right) n \le \overlap(\hat{\Pi}, \id) \le n-2 , \calE_1, \calE_2} \\
& \le \sum_{n_1 \ge (1-\eta) n}^{n-2} \sum_{\Pi \in \calS(n_1) }  \expect{\exp\left( - \frac{1}{32\sigma^2} \norm{\PPiXperp(X\beta_*)}^2\right)} \\
& \le n^{-\epsilon/8},
\end{align*}
where the last inequality holds by~\prettyref{lmm:MGF_bound_2} and the assumption that $\sigma/\norm{\beta_\ast} \le n^{-2-\epsilon}$. 
Finally, applying~\prettyref{lmm:noise_projection}
and~\prettyref{lmm:signal_projection}, we arrive at
\begin{align*}
\prob{  \left( 1- \eta \right)n \le \overlap(\hat{\Pi}, \id) \le n-2} 
& \le \prob{\left( 1- \eta \right) n \le \overlap(\hat{\Pi}, \id) \le n-2 , \calE_1, \calE_2} \\
& \quad + \prob{\calE_1^c} +\prob{\calE_2^c}\\
& \le 6n^{-\epsilon/8}.
\end{align*}


\section{Conclusions and open problems}
In this paper we resolved the information-theoretically optimal thresholds for exactly and almost exactly recovering the unknown permutation in shuffled linear regression with random design in the sublinear regime of $d=o(n)$.
In addition to determining the sharp threshold in the linear regime of $d=\Theta(n)$ mentioned in \prettyref{sec:results}, a few other problems remain outstanding.

First, the estimator \prettyref{eq:MLE1} attaining the sharp thresholds involves solving the computationally expensive QAP problem. Although for low dimensions this can be approximately computed in $n^{O(d)}$ time, the resulting algorithm is far from practical as it involves searching over an $\delta$-net in $d$ dimensions. For $d\to\infty$, currently there is no polynomial-time algorithms except in the special case of $\sigma=0$ \cite{hsu2017linear,andoni2017correspondence}.

Second, it is of interest to extend the current results to multivariate responses where 
each response $y_i$ is $m$-dimensional for $m>1$. In other words,  
$y = \Pi_* X \beta_* + w$, where 
$\beta_*\in\reals^{d \times m}$.
This has been considered in several existing works such as \cite{pananjady2017linear,zhang2020optimal}, where it is observed that multiple responses can significantly reduce the  required SNR. Drawing from existing results on related models in LAP and QAP \cite{kunisky2022strong,WWXY22}, we conjecture that the optimal thresholds for exact and almost exact recovery are given by
$\SNR = n^{4/m}$ and 
$ n^{2/m}$, respectively, provided that $m$ is not too large.
While one can deduce the lower bound from that in \cite{WWXY22} by considering the oracle setting of a known $\beta_*$, analyzing the counterpart of 
\prettyref{eq:MLE1} remains open.

\section*{Acknowledgment}
 Y.~Wu is supported in part by the NSF Grant CCF-1900507 and an Alfred Sloan fellowship. J.~Xu is supported in part by an NSF CAREER award CCF-2144593.

\bibliographystyle{alpha}
\bibliography{shuffled,IEEEabrv}

\newpage
\appendices

\section{Additional proofs}

\subsection{Moment generating function and cycle decomposition}
\label{app:cycle}

For $k\in[n]$, let $n_k(\Pi)$ denote the number of $k$-cycles in the cycle decomposition of the permutation $\Pi$. 
Let $\fc(\Pi)=\sum_{k=1}^n n_k(\Pi)$ denote the total number of cycles of $\Pi.$

\begin{lemma} \label{lem:mgf}
Fix a permutation $\Pi$ and $\beta \in\reals^d$.
Define 
\begin{align}
    p &\triangleq \frac{1}{2} \left( \sqrt{1 + 2t \lVert \beta_* + \beta \rVert_2^2} + \sqrt{1 + 2t \lVert \beta_* - \beta \rVert_2^2} \right), \label{eq:p} \\
    q &\triangleq \frac{1}{2} \left( \sqrt{1 + 2t \lVert \beta_* + \beta \rVert_2^2} - \sqrt{1 + 2t \lVert \beta_* - \beta \rVert_2^2} \right). \label{eq:q}
\end{align}
Then
\begin{align}
    \E_X \left[ \exp \left( -t \lVert X \beta_* - \Pi X \beta \rVert_2^2 \right) \right]
    &= \prod_{k=1}^{n} (p^k - q^k)^{-n_k(\Pi)} \label{eq:mgf1} \\
    &\leq  (1 + 2t \lVert \beta - \beta_* \rVert_2^2 )^{-\frac{n_1}{2}} ( c_0 \sqrt{ t} \lVert \beta_* \rVert_2)^{-(n-\fc(\Pi))} \label{eq:mgf3}Z\\
    &\leq  (1 + 2t \lVert \beta - \beta_* \rVert_2^2 )^{-\frac{n_1}{2}} ( c_0 \sqrt{t} \lVert \beta_* \rVert_2)^{-\frac{n-n_1(\Pi)}{2}},\label{eq:mgf2} 
\end{align}
where $c_0>0$ is an absolute constant, and the last two inequalities hold under the additional assumption that $\sqrt{t} \|\beta_*\|_2\geq C_0$ for some universal constant $C_0>0$.

\end{lemma}


\begin{proof}
We first prove the equality \prettyref{eq:mgf1}.
Note that $X \beta_* \sim \NN(0, \norm{\beta_*}^2 I_n)$.
Let us decompose $\beta$ along $\beta_*$: $\frac{\beta}{\norm{\beta}} = \cos \theta \frac{\beta_*}{\norm{\beta_*}} + \sin \theta \beta_*^\perp$, where $\cos \theta = \frac{\langle \beta_*, \beta \rangle}{\norm{\beta_*} \norm{\beta}}$, and $\beta_*^\perp$ is a unit vector orthogonal to $\beta_*$.
Defining $Z_1 \triangleq X \frac{\beta_*}{\norm{\beta_*}}$ and $Z_2 \triangleq X \beta_*^\perp$, we have that $Z_1, Z_2 \overset{\text{iid}}{\sim} \mathcal{N}(0, I_n)$.
So, we can rewrite $X \beta_* = \norm{\beta_*} Z_1$ and $X \beta = \norm{\beta} ( \cos \theta Z_1 + \sin  \theta  Z_2 )$.
Therefore, 
\begin{align*}
\norm{X\beta_* - \Pi X \beta}^2
& =  \norm{ \norm{\beta_*} Z_1 - \norm{\beta} \left( \cos \theta \Pi Z_1
+ \sin \theta \Pi Z_2\right) }^2\\
&= \norm{\beta}^2 \sin^2 \theta \norm{ \Pi Z_2 + \left(\tfrac{\cos \theta}{\sin \theta} \Pi -\tfrac{\norm{\beta_*}}{\norm{\beta} \sin \theta} I \right) Z_1 }^2.
\end{align*}
Let $Z \triangleq \Pi Z_2$, $H \triangleq \frac{\cos \theta}{\sin \theta} \Pi - \frac{\norm{\beta_*}}{\norm{\beta} \sin \theta} I$, $M \triangleq H Z_1$, and $t' = -t \norm{\beta}^2 \sin^2 \theta$. Note that the distribution of $Z$ does not depend on $\Pi$; hence $Z$ is independent of $M$.
Then,
\begin{align*}
    (\dagger) &\triangleq \E_X \left[ \exp \left( -t \lVert X \beta_* - \Pi X \beta \rVert_2^2 \right) \right] 
    = \E_M\left[ \E_{Z \mid M} \left[ \exp\left( t' \norm{ Z + M }^2 \right) \right] \right].
\end{align*}
We recognize the inner expectation as the MGF for a noncentral $\chi^2(n)$ distribution with noncentrality parameter $\lambda = \norm{M}^2$.
Thus,
\begin{align*}
    (\dagger) 
    &= \E_M\left[ ( 1 - 2t' )^{-\frac{n}{2}} \exp\left( \tfrac{t' \norm{M}^2}{1-2t'} \right) \right] 
    = ( 1 - 2t' )^{-\frac{n}{2}} \E_{Z_1} \left[ \exp\left( \tfrac{t'}{1-2t'} \norm{H Z_1}^2 \right) \right].
\end{align*}
Using the expression of MGF for quadratic forms of normal random variables~\cite[Lemma 2]{baldessari1967distribution}, we deduce that
\begin{align*}
    \E_{Z_1} \left[ \exp\left( \tfrac{t'}{1-2t'} \norm{H Z_1}^2 \right) \right]
    &= \left( \det\left( I - \tfrac{2t'}{1-2t'} H^\top H \right) \right)^{-\frac{1}{2}}.
\end{align*}
The eigenvalues of $H$ are $\lambda_j(H) = \tfrac{\cos \theta}{\sin \theta} \lambda_j(\Pi) - \tfrac{\norm{\beta_*}}{\norm{\beta} \sin \theta}$, 
and the eigenvalues of $I - \tfrac{2t'}{1-2t'} H^\top H$ are $ 1 - \tfrac{2t'}{1-2t'} |\lambda_j(H)|^2$.
Therefore, 
$$
\E_{Z_1} \left[ \exp\left( \tfrac{t'}{1-2t'} \norm{H Z_1}^2 \right) \right] = \prod_{j=1}^{n}  \left(1 - \tfrac{2t'}{1-2t'} |\lambda_j(H)|^2 \right)^{-\frac{1}{2}}.
$$
Thus,
\begin{align*}
    (\dagger) 
    &= (1 - 2t')^{-\frac{n}{2}} \prod_{j=1}^{n}  \left(1 - \tfrac{2t'}{1-2t'} |\lambda_j(H)|^2 \right)^{-\frac{1}{2}} \\
    &=  \prod_{j=1}^{n} \left(1 -2t'-2t'|\lambda_j(H)|^2 \right)^{-\frac{1}{2}} \\
    &= \prod_{j=1}^{n}  \Bigg(1 + 2t \norm{\beta}^2 \sin^2\theta + 2t \Big| \norm{\beta} \cos \theta \lambda_j(\Pi) 
    - \norm{\beta_*}\Big|^2 \Bigg)^{-\frac{1}{2}}.
\end{align*}

The eigenvalues of the permutation matrix $\Pi$ are directly determined by its cycle decomposition.
Each $k$-cycle is associated with $k$ eigenvalues of the form $\exp(\tfrac{2\pi i}{k} j_k)$ for $0 \leq j_k \leq k-1$, where $i$ here denotes the imaginary unit. 
Thus, each $\exp(\tfrac{2\pi i}{k} j_k)$ has multiplicity $n_k$, the number of $k$-cycles in $\Pi$.
Using the expansion,
\begin{align*}
    \Big| \norm{\beta} \cos \theta e^{\tfrac{2\pi i}{k}j_k} - \norm{\beta_*} \Big|^2
    = \norm{\beta}^2 \cos^2 \theta + \norm{\beta_*}^2 - 2 \norm{\beta_*} \norm{\beta} \cos \theta \cos \tfrac{2\pi j_k}{k},
\end{align*}
we can write
\begin{align}
    (\dagger) 
    &= \prod_{k=1}^{n}  \big(1 + 2t \norm{\beta}^2 \sin^2(\theta) + 2t ( \norm{\beta}^2 \cos^2 \theta + \norm{\beta_*}^2 - 2 \norm{\beta_*} \norm{\beta} \cos \theta \cos \tfrac{2\pi j_k}{k} \big)^{-\frac{n_k}{2}} \notag \\
    &= \prod_{k=1}^{n} \Bigg[ \prod_{j=0}^{k-1} \big(1 + 2t \norm{\beta}^2 + 2t \norm{\beta_*}^2 - 4t\norm{\beta_*} \norm{\beta} \cos \theta \cos \tfrac{2\pi j}{k} \big) \Bigg]^{-\frac{n_k}{2}}. \label{eq:MGF_explicit}
\end{align}
To simplify this product, let us define a shorthand for the inner product in \prettyref{eq:MGF_explicit}: 
$$
f_k(a,b,\theta) \triangleq \prod_{j=0}^{k-1} (1 + 2t (a^2 + b^2) - 4t a b \cos \theta \cos \tfrac{2\pi j}{k}).
$$
To further simplify $f_k$, let us introduce
\begin{align*}
    p &= \frac{1}{2} \left( \sqrt{ 1 + 2t(a^2+b^2) + 4tab\cos(\theta)} + \sqrt{ 1 + 2t(a^2+b^2) - 4tab\cos(\theta) } \right), \\
    q &= \frac{1}{2} \left( \sqrt{ 1 + 2t(a^2+b^2) + 4tab\cos(\theta)} - \sqrt{ 1 + 2t(a^2+b^2) - 4tab\cos(\theta) } \right).
\end{align*}
Now, we can write $f_k$ as
\begin{align*}
    f_k(a,b,\theta) &= \prod_{j=0}^{k-1} (p^2 + q^2 - 2pq \cos(\tfrac{2\pi j}{k})) 
    = p^{2k} + q^{2k} - 2 p^k q^k
    = (p^k - q^k)^2.
\end{align*}
Plugging in $a = \norm{\beta_*}$ and $b = \norm{\beta}$,
\begin{align*}
    (\dagger) &= \prod_{k=1}^{n} [f_k(a,b,\theta)]^{-\frac{n_k}{2}}
     = \prod_{k=1}^{n}(p^k - q^k)^{-n_k}.
\end{align*}
Recall that $\cos \theta = \langle \beta_*, \beta \rangle/ (\norm{\beta_*} \norm{\beta})$.
Thus, 
\begin{align*}
1 + 2t(a^2+b^2) \pm 4tab\cos(\theta)
&= 1 + 2t\left(\norm{\beta_*}^2+\norm{\beta}^2 \pm  2 \langle \beta_*, \beta \rangle\right)
 = 1 + 2t \norm{\beta_* \pm \beta}^2,
\end{align*}
yielding $p$ and $q$ as stated in \prettyref{eq:p}
and \prettyref{eq:q}.



    Next, we prove the upper bounds \prettyref{eq:mgf3} and \prettyref{eq:mgf2}.
We isolate the special case of $k=1$, noting that $p-q = \sqrt{1 + 2t \lVert \beta - \beta_* \rVert_2^2}$.
For $k \geq 2$, using~\prettyref{lmm:pkqk} below, we have that
$p^k - q^k \ge (c_0 \sqrt{t} \norm{\beta_*})^{k-1}$. 
Therefore,
\begin{align*}
    \prod_{k=1}^{n} (p^k - q^k)^{-n_k(\Pi)}
    &\leq (p - q)^{-n_1(\Pi)} (c_0\sqrt{t} \lVert \beta_* \rVert_2)^{-\sum_{k=2}^{n} n_k(\Pi) (k-1)} \\
    &\leq (p - q)^{-n_1(\Pi)} (c_0 \sqrt{t} \lVert \beta_* \rVert_2)^{-(n-\fc(\Pi))}\\
    &\leq (p - q)^{-n_1(\Pi)} (c_0\sqrt{t} \lVert \beta_* \rVert_2)^{-\frac{n-n_1(\Pi)}{2}},
\end{align*}
where the last inequality applies $\sum_{k=1}^{n} n_k(\Pi) \geq n_1(\Pi)+\frac{n-n_1(\Pi)}{2} =\frac{n+n_1(\Pi)}{2}$.
\end{proof}

\begin{lemma}
    \label{lmm:pkqk}
Let $p$ and $q$ be defined in \prettyref{eq:p}--\prettyref{eq:q}. 
There exist absolute constants $c_0$, such that if $\sqrt{t} \|\beta_*\|_2\geq 5 \sqrt{2}$, then for all $k\geq 2$,
\[
p^k-q^k \geq (c_0\sqrt{t} \|\beta_*\|_2)^{k-1}.
\]
\end{lemma}
\begin{proof}
Rescaling for convenience, let $u_*\equiv\sqrt{2t} \beta_*$ and 
$u\equiv\sqrt{2t} \beta$.
    Note that $p,q$ only depends on $u \pm u_*$.
    Without loss of generality, by rotation we can assume that $u_*= \norm{u_*} e_1$, where $e_1=(1,0,\ldots,0)$.
    Let $u_1=a\in\reals$ and $b=\norm{u_{\backslash 1}}^2$.
    Assume that $s\equiv \norm{u_*} = \sqrt{2t} \|\beta_*\|_2 \geq 10$.
    Then we have
$$
p=\frac{1}{2}(\sqrt{1+(s+a)^2+b}+\sqrt{1+(s-a)^2+b}) \text{ and } q=\frac{1}{2}(\sqrt{1+(s+a)^2+b}-\sqrt{1+(s-a)^2+b}).
$$

We first show that for any fixed $a$, 
$p^k-q^k$ is minimized at $b=0$. 
This follows from the identity that 
\[
\frac{d}{db}(p^k-q^k)
= \frac{k(p^k+q^k)}{2 \sqrt{1+(s+a)^2+b} \sqrt{1+(s-a)^2+b}} \geq 0.
\]
Next, setting $b=0$, let 
\begin{align*}
    g(a) &  \triangleq 
\pth{\frac{1}{2}(\sqrt{1+(s+a)^2}+\sqrt{1+(s-a)^2})}^k - 
\pth{\frac{1}{2}(\sqrt{1+(s+a)^2}-\sqrt{1+(s-a)^2})}^k.
\end{align*}
To lower bound $g(a)$, consider three cases:

\textit{Case I: $a\leq 0$.}
We only need to consider this case for odd $k$ (for otherwise  $g(a)$ is an even function in $a$). Then the second term in $g(a)$ is nonpositive.
Furthermore, by convexity of $x \mapsto \sqrt{1+x^2}$,
$$
p = \frac{1}{2}(\sqrt{1+(s+a)^2}+\sqrt{1+(s-a)^2})
\geq \sqrt{1+( \tfrac{s+a}{2} + \tfrac{s-a}{2} )^2} = \sqrt{1 + s^2}.
$$
So 
$g(a)
\geq (\sqrt{1+s^2})^k \geq s^k$.

\textit{Case II:} $0 \leq a\leq s/2$. In this case, $0 \le q \le p$. Furthermore,
\begin{align*}
  \frac{p}{q}  &=
  \frac{
  \sqrt{1+(s+a)^2}+\sqrt{1+(s-a)^2}
  }{
  \sqrt{1+(s+a)^2}-\sqrt{1+(s-a)^2}
  }
  = 
  \frac{ (\sqrt{1+(s+a)^2}+\sqrt{1+(s-a)^2})^2}{
  4as
  }
  \ge \frac{a^2+s^2}{2as}
  \ge \frac{5}{4},
\end{align*}
where the last inequality holds because 
$(a^2+s^2)/(2as)$ is monotone decreasing in $a$ for $0 \le a \le s/2.$
Therefore, 
$$
g(a) \ge p^k (1- (4/5)^k) \ge p^k /5 \ge (s/2)^k/5,
$$
where the last inequality holds because $p \ge (s+a)/2 \ge s/2.$

\textit{Case III:}
$a \geq s/2$.
In this case. $0 \le q \le p$ and $p-q\ge 1$ by definition. Thus, by the convexity of $x^k$, 
$p^k-q^k \ge kq^{k-1}(p-q) \ge kq^{k-1}. $
Furthermore, 
\begin{align*}
  q  = & \frac{(s+a)^2-(s-a)^2}{2(\sqrt{1+(s+a)^2}+\sqrt{1+(s-a)^2})}
\geq  \frac{sa}{\sqrt{1+(s+a)^2} } 
 \geq \frac{sa}{s+a} \geq s/3,
\end{align*}
Therefore, we have $g(a) \ge k (s/3)^{k-1}.$

Assembling all three cases, we conclude that 
$$
g(a) \ge \min \left\{ (s/2)^k/5, k (s/3)^{k-1} \right\}
\ge (s/3)^{k-1},
$$
where the last inequality holds under the assumption that $s\ge 10$. 
%
%
\end{proof}

\subsection{Proof of \prettyref{lem:discretization}}
The fact that $|N_\delta| \leq (1+\frac{2r}{\delta})^d$ follows from the standard volume bound (see, e.g., \cite[Corollary 4.2.13]{Vershynin_2018}).

Now, we prove the approximation inequality.
Let $\beta' \in N_{\delta}$ satisfy $\lVert \hat{\beta}_{\Pi} - \beta' \rVert_2 \leq \delta$.
Since $\lVert \hat{\beta}_{\Pi} \rVert_2 < r$, the minimizer for $\beta \in B_r(\beta_*)$ is equal to $\hat{\beta}_{\Pi}$, the minimizer over $\R^d$.
Thus, by the orthogonal projection theorem, $y - \Pi X \hat{\beta}_{\Pi}$ is orthogonal to the column space of $\Pi X$, so $\langle y - \Pi X \hat{\beta}_{\Pi}, \Pi X (\beta' - \hat{\beta}_{\Pi}) \rangle = 0$.
Therefore,
\begin{align*}
    \norm{ y - \Pi X \beta' }^2
    &= \lVert y - \Pi X \hat{\beta}_{\Pi} \rVert^2 + \lVert \Pi X (\beta' - \hat{\beta}_{\Pi}) \rVert_2^2 - 2 \langle y - \Pi X \hat{\beta}_{\Pi}, \Pi X (\beta' - \hat{\beta}_{\Pi}) \rangle \\
    &\leq \lVert y - \Pi X \hat{\beta}_{\Pi} \rVert_2^2 + \opnorm{ \Pi X } \lVert \beta' - \hat{\beta}_{\Pi} \rVert_2^2 \\
    &= \lVert y - \Pi X \hat{\beta}_{\Pi} \rVert_2^2 + \opnorm{ X } \delta^2.
\end{align*}
Taking the minimum w.r.t. $\beta'$ yields the desired result.


\subsection{Proof of \prettyref{lem:events}}

Let $\sigma_{\max}(X)$ and $\sigma_{\min}(X)$ denote the maximum and minimum singular values of $X$, which is $n\times d$ and independent $\calN(0,1)$ entrywise. By \cite[Theorem II.7]{MR1863696}, for $n>d$, we have
\begin{align}
    \prob{\sigma_{\max}(X) > \sqrt{n}+\sqrt{d}+t} & \leq e^{-t^2/2} \label{eq:sval_max} \\ 
    \prob{\sigma_{\min}(X) < \sqrt{n}-\sqrt{d}-t} & \leq e^{-t^2/2}. \label{eq:sval_min}
\end{align}
The following tail bounds for the chi-squared distribution will also be useful. Let $Z_m \sim \chi^2(m)$. Then,
\begin{align}
\prob{Z_m \ge m + 2\sqrt{ m t} + 2t} & \le \exp(-t), \label{eq:chi_lower} \\
\prob{Z_m \le m - 2\sqrt{ m t}} & \le \exp(-t). \label{eq:chi_upper}
\end{align}

We consider $\EE_1$ first. 
Let $\sigma_0^2\triangleq n^{-2-\epsilon}\norm{\beta_\ast}^2$. By union bound,  
\begin{align*}
\prob{\calE_1^c}
& \le |N_\delta(r)|
\sum_{n_1=0}^{(1-\eta)n}
\sum_{\Pi \in \calS(n_1)}
\prob{ 
\norm{X\beta_\ast - \Pi X \beta }^2 \le \sigma^2_0 (n-n_1)}.
\end{align*}
For a fixed $\Pi$ and $\beta$, applying the Chernoff bound, we deduce that
\begin{align*}
\prob{ 
\norm{X\beta_\ast - \Pi X \beta }^2 \le \sigma^2_0 (n-n_1)} 
& \le e^{(n-n_1)/32}
\expect{e^{-
\frac{1}{32\sigma_0^2}\norm{X\beta_\ast - \Pi X \beta }^2 } }.
\end{align*}
By~\prettyref{lem:discretization} and the assumption that $r/\delta \le n^2$, $|N_\delta(r)| \le (1+2r/\delta) \le (1+2n^2)^d$. It follows that
\begin{align*}
\prob{\calE_1^c}
& \le |N_\delta(r)|
\sum_{n_1=0}^{(1-\eta)n}
e^{(n-n_1)/32}
\sum_{\Pi \in \calS(n_1)}
\expect{e^{-
\frac{1}{32\sigma_0^2}\norm{X\beta_\ast - \Pi X \beta }^2 } }\\
&  \le (1+2 n^{2})^d n^{-\epsilon \eta n/10}
\le n^{-\epsilon \eta n/20},
\end{align*}
where we applied~\prettyref{lmm:MGF_bound_1} and used the assumptions that $\eta \epsilon n \ge 100d$ and $\eta \ge n^{-\epsilon/10}.$

Next, we consider $\EE_2$.  Applying \prettyref{eq:sval_max} with $t = \sqrt{n} $,
we get that
$$
\prob{\opnorm{X} \le 3\sqrt{n}} \le e^{-n/2}.
$$


Finally, we show $\EE_3$ holds with high probability.
By the definition of $\hat{\beta}_\Pi$ given in~\prettyref{eq:beta_Pi}, 
\begin{align*}
\norm{\hat{\beta}_\Pi}
&\le \norm{(X^\top X)^{-1} X^\top \Pi^\top y} 
\le \opnorm{(X^\top X)^{-1} X^\top} \norm{y}. 
\end{align*}
Since $\lVert (X^\top X)^{-1} X^\top \rVert_{\rm op} = (\sigma_\text{min}(X))^{-1}$, it follows from \prettyref{eq:sval_min} that 
with probability at least $1-\exp(-n/32),$
$$
\lVert (X^\top X)^{-1} X^\top \rVert_{\rm op} \le  
\frac{1}{\sqrt{n}-\sqrt{d} - \sqrt{n}/4} \le \frac{4}{\sqrt{n}},
$$
where the last inequality holds due to the assumption that $d \le n/4$.
Moreover, since $y=X\beta + w$, it follows that $\norm{y}^2 \sim (\norm{\beta_*}^2+\sigma^2) \chi^2 (n).$ Therefore, by~\prettyref{eq:chi_lower}, with probability at least $1-\exp(-n)$,
$\norm{y}^2 \le 5(\norm{\beta_*}^2+\sigma^2)n$. Hence,
 we get that 
 $$
\norm{\hat{\beta}_\Pi} \le 4 \sqrt{5(\norm{\beta_*}^2+\sigma^2)n}
\le 4\sqrt{10} \norm{\beta_*},
 $$
 where the last inequality holds under the assumption that $\SNR=\norm{\beta_*}^2/\sigma^2 \ge 1$. Finally, we get that
 $$
 \norm{\hat{\beta}_{\Pi} - \beta_*}
 \le \norm{\hat{\beta}_{\Pi} } +\norm{\beta_*}
 \le \left( 4\sqrt{10}+1 \right) \norm{\beta_*}.
 $$

\subsection{Proof of~\prettyref{lmm:MGF_bound_1}}
\label{app:MGF_bound_1}

By Lemma~\ref{lem:mgf}, we have 
\begin{align*}
\expect{\exp\left( - \frac{1}{32\sigma^2} \norm{X\beta_*-\Pi X\beta }^2 \right) }
& \le \left( c_0 \norm{\beta_\ast} /\sigma \right)^{-(n-\fc(\Pi))}  \le \left( c_0 \norm{\beta_\ast} /\sigma \right)^{-(n-n_1-\fc(\tilde{\Pi}))},
\end{align*}
where $c_0$ is an absolute constant, $\fc(\Pi)$ is the total number of cycles of $\Pi$ and $\tilde{\Pi}$ is the restriction of permutation $\Pi$ on its non-fixed points, which, by definition, is a derangement. 
Recall $\SNR= \norm{\beta_\ast}^2/\sigma^2$.
It follows that 
\begin{align*}
\sum_{\Pi \in \calS(n_1)} \expect{\exp\left( - \frac{1}{32\sigma^2} \norm{X\beta_*-\Pi X\beta }^2 \right) }
& \le \left( c_0^2\SNR \right)^{-\frac{n-n_1}{2}}
\sum_{\Pi \in \calS(n_1)} \left( c_0 \sqrt{
\SNR}\right)^{\fc(\tilde{\Pi})} \\
& \le  \left( c_0^2 \SNR \right)^{-\frac{n-n_1}{2}}
\binom{n}{n_1} (n-n_1)!
\left( \frac{16 c_0 \sqrt{\SNR}}{n-n_1}\right)^{\frac{n-n_1}{2}},
\end{align*}
where the last inequality follows from~\cite[eq.~(49)]{WWXY22} applying the moment generating of the number of cycles in a random derangement.
Therefore,
\begin{align*}
& \sum_{n_1=0}^{(1-\eta)n} C^{n-n_1} \sum_{\Pi \in \calS(n_1)} 
\expect{\exp\left( - \frac{1}{32\sigma^2} \norm{X\beta_*-\Pi X\beta }^2 \right) } \\
& \le \sum_{n_1=0}^{(1-\eta)n} \binom{n}{n_1} (n-n_1)!  \left( \SNR \right)^{-\frac{n-n_1}{2}}
\left( \frac{C' \sqrt{\SNR}}{n-n_1}\right)^{\frac{n-n_1}{2}} \\
& \le \sum_{n_1=0}^{(1-\eta)n} \left( 
\frac{C' n }{\eta \sqrt{\SNR}}\right)^{\frac{n-n_1}{2}}  \le n^{-\epsilon \eta n /10},
\end{align*}
where $C'$ is an absolute constant; the last inequality holds for all sufficiently large $n$ and follows from the assumptions that $\SNR \ge n^{2+\epsilon}$
and $\eta \ge n^{-\epsilon/10}$.


\subsection{Proof of~\prettyref{lmm:noise_projection}}
While $\norm{\PPiX(w)}^2 \sim \sigma^2 \chi^2(d)$ and $\norm{\PX(w)}^2 \sim \sigma^2 \chi^2(d)$, it turns out that
\begin{align}
\norm{\PPiX(w)}^2 - \norm{\PX(w)}^2 = \sigma^2 \left( Z_m - \tilde{Z}_m \right), \label{eq:chi_square_rep}
\end{align}
where $m \le \min\{d, n-n_1\}$, and  $Z_m$ and  $\tilde{Z}_m$ are two (possibly correlated) chi-squared random variables with $m$ degree of freedom (see~\cite[Proof of Claim (23a)]{pananjady2017linear} for a proof).\footnote{While both our proof and the proof of Claim (23a) in~\cite{pananjady2017linear} rely on~\prettyref{eq:chi_square_rep}, our use of tail bounds  on $Z_m$ and $\tilde{Z}_m$ differs from those in~\cite{pananjady2017linear}. This is because the proof in~\cite{pananjady2017linear} uses subexponential tail bounds on $Z_m$ and $\tilde{Z}_m$, which are only valid when $m$ is a constant. } 
Using the tail bounds of chi-squared distribution in \prettyref{eq:chi_lower} and \prettyref{eq:chi_upper} and choosing $t=2(n-n_1) \log n$, it follows that with probability at most $2n^{-2(n-n_1)}$, 
\begin{align*}
\norm{\PPiX(w)}^2 - \norm{\PX(w)}^2 
& \le 4 \sigma^2 \sqrt{2m(n-n_1) \log n } + 4 \sigma^2 (n-n_1) \log n \\
& \le 10 \sigma^2 (n-n_1) \log n.
\end{align*}
The proof is complete by applying a union bound over $\calS(\pi)$ and $0 \le n_1 \le n-2$ and invoking $|\calS(\pi)| \le n^{n-n_1}.$

\subsection{Proof of~\prettyref{lmm:signal_projection}}
Let $\sigma_0^2/\norm{\beta_\ast}=n^{-4-\epsilon}$.
The proof follows from combining the Chernoff bound with~\prettyref{lmm:MGF_bound_2}. 
In particular, by the Chernoff bound, for any $t\ge0,$
\begin{align*}
\prob{\norm{\PPiXperp(X\beta_*)}^2 \le \sigma_0^2 (n-n_1) } 
& \le e^{t \sigma_0^2(n-n_1)}
\expect{e^{-t\norm{\PPiXperp(X\beta_*)}^2}}
\end{align*}
Choosing $t=\frac{1}{32\sigma_0^2}$ yields that
\begin{align*}
\prob{\norm{\PPiXperp(X\beta_*)}^2 \le \sigma_0^2 (n-n_1) } 
& \le \exp\left( \frac{n-n_1}{32} \right)
\expect{\exp \left(-\frac{1}{32\sigma_0^2}\norm{\PPiXperp(X\beta_*)}^2\right)}.
\end{align*}
By a union bound,
\begin{align*}
\prob{\calE_2^c} 
& \le \sum_{n_1 \ge (1-\eta)n}^{n-2}
e^{ (n-n_1)/32}
\sum_{\pi\in \calS(n_1)} 
\expect{e^{-\frac{1}{32\sigma_0^2}\norm{\PPiXperp(X\beta_*)}^2}}
\le n^{-\epsilon/8 },
\end{align*}
where the last inequality holds by~\prettyref{lmm:MGF_bound_2}.
\subsection{Proof of~\prettyref{lmm:MGF_bound_2}}
Our proof utilizes the following crucial fact shown in~\cite[Proof of Lemma 5]{pananjady2017linear}:
$$
\norm{\PPiXperp(X \beta_*)} \overset{(d)}{=} \norm{\beta_*}
\xi \norm{\calP_{(\Pi x_1)^\perp}(x_1)},
$$
where the equality holds in distribution, $\xi$ denotes the length of a unit vector in $(n-1)$-dim space projected onto a $(n-d)$-dim subspace chosen uniformly at random, $x_1$ is the first column of $X$, and $\xi$ is independent of $x_1$ and $\Pi.$
Therefore, 
\begin{align*}
\expect{\exp\left( - \frac{1}{32\sigma^2} \norm{\PPiXperp(X\beta_*)}^2\right)} 
& = \expect{\exp\left( - \frac{\norm{\beta_*}^2\xi^2}{32 \sigma^2} \norm{\calP_{(\Pi x_1)^\perp}(x_1)}^2 \right)} \\
& \le \expect{\exp\left( - \frac{\norm{\beta_*}^2 \tau^2}{32\sigma^2} \norm{\calP_{(\Pi x_1)^\perp}(x_1)}^2 \right)} 
+ \prob{\xi^2 \le \tau^2},
\end{align*}
where $\tau>0$ can be any threshold. 

We use the following concentration inequality for $\xi$~\cite[Lemma 2.2]{dasgupta2003elementary}: for all $t <1$, 
$$
\prob{\xi^2 \le t \frac{n-d}{n-1}}
\le t^{(n-d)/2} \left( 
1+ \frac{(1-t)(n-d)}{(d-1)} \right)^{(d-1)/2}
$$
Setting $t=n^{-\epsilon/2}$, we deduce that
\begin{align*}
\prob{\xi^2 \le n^{-\epsilon/2} \frac{n-d}{n-1}}
& \le n^{-\epsilon(n-d)/4} \left( 
\frac{n-1}{d-1} \right)^{(d-1)/2} 
\le n^{-\epsilon n/5},
\end{align*}
where the last inequality holds for all sufficiently large $n$ due to the assumption that $d=o(n).$
In this way, we obtain a tail probability $n^{-\epsilon n/5}$, so that we can still get a vanishing probability after taking a union bound over $\Pi \in \calS(n_1)$ for $n_1 \ge (1-\epsilon/8)n.$ Thus, we choose $\tau^2=n^{-\epsilon/2} \frac{n-d}{n-1}$. 

Note that a similar concentration bound of $\xi^2$ was used in~\cite[Proof of Lemma 5]{pananjady2017linear}. However, since the argument in~\cite{pananjady2017linear} directly applies a union bound over all $\Pi \neq \Pi_*$, $\tau^2$ was chosen to be a much smaller value $n^{-2n/(n-d)}$. As we will see, this will lead to a much more stringent condition on $\sigma$.

Next, observe that 
\begin{align*}
\norm{\calP_{(\Pi x_1)^\perp}(x_1)}^2 
& =\norm{x_1}^2 - \iprod{x_1}{\Pi x_1}^2\norm{x_1}^{-2} \\
&\ge \norm{x_1}^2- \left|\iprod{x_1}{\Pi x_1}\right| \\
&=\min\left\{  \norm{x_1 - \Pi x_1}^2, \norm{x_1+\Pi x_1}^2 \right\}.
\end{align*}
Therefore, 
\begin{align*}
\expect{\exp\left( - \frac{\norm{\beta_*}^2\tau^2}{32\sigma^2} \norm{\calP_{(\Pi x_1)^\perp}(x_1)}^2 \right)}  
& \le  \expect{\exp\left( - \frac{\norm{\beta_*}^2\tau^2}{32\sigma^2} \min\left\{  \norm{x_1 - \Pi x_1}^2, \norm{x_1+\Pi x_1}^2 \right\} \right)} \\
& \le \sum_{s = \pm 1} \expect{\exp\left( - \frac{\norm{\beta_*}^2\tau^2}{32\sigma^2}   \norm{x_1 - s \Pi x_1}^2 \right)}
\\
& \le 2 \left(\frac{C_0\sigma}{ \norm{\beta_*}\tau}\right)^{(n-n_1)/2},
\end{align*}
where $C_0$ is some absolute constant, and the last inequality follows from our MGF bound~\prettyref{eq:mgf2}.
In conclusion, we get that
\begin{align*}
\expect{\exp\left( - \frac{1}{32\sigma^2} \norm{\PPiXperp(X\beta_*)}^2\right)} 
& \le 2 \left(\frac{C_0\sigma}{ \norm{\beta_*}\tau}\right)^{(n-n_1)/2}
+ n^{-\epsilon n/5}.
\end{align*}
Therefore, 
\begin{align*}
& \sum_{n_1 \ge (1-\eta) n}^{n-2} C^{n-n_1} \sum_{\Pi \in \calS(n_1) }  \expect{\exp\left( - \frac{1}{32\sigma^2} \norm{\PPiXperp(X\beta_*)}^2\right)}\\
& \le
\sum_{n_1 \ge (1-\eta) n}^{n-2}
(Cn)^{n-n_1}
\left(2\left(\frac{C_0\sigma}{ \norm{\beta_*}\tau}\right)^{(n-n_1)/2}
+ n^{-\epsilon n/5}  \right) \\
& \le 
\sum_{n_1 \ge (1-\eta) n}^{n-2}
\left(2 \left(\frac{C^2 n^2 C_0\sigma}{ \norm{\beta_*}\tau}\right)^{(n-n_1)/2}
+ (Cn)^{\eta n}  n^{-\epsilon n/5}   \right) \\
& \le n^{-\epsilon/8 },
\end{align*}
where the last inequality holds for all sufficiently large $n$ by the assumptions that $\sigma/\norm{\beta_*} \le n^{-2-\epsilon/2}$, $\eta \le \epsilon/8$, and $d=o(n).$

\end{document}